\documentclass[11pt]{amsart}
\usepackage{amsmath}
\usepackage{amssymb}
\usepackage{graphicx}
\usepackage{url}
\usepackage{multirow}

\usepackage{algorithm}
\usepackage{algorithmic}
\usepackage{float}
\usepackage{comment}
\usepackage[font=small,labelfont=bf]{caption}
\usepackage{subcaption}

\usepackage{color}
\usepackage{hyperref}
\usepackage{multicol}

\newtheorem{theorem}{Theorem}[section]

\newtheorem{lemma}[theorem]{Lemma}

\theoremstyle{definition}

\newtheorem{example}[theorem]{Example}
\newtheorem{remark}[theorem]{Remark}

\numberwithin{equation}{section}

\title[Lipschitz-Free Mirror Descent with Absolute and Relative Inexactness] {Lipschitz-Free Mirror Descent Methods for Relatively Strongly Convex Functions with/without Absolute and Relative Inexactness}

\author[M.~S.~Alkousa]{Mohammad S. Alkousa}
\address[M.~S.~Alkousa]{Innopolis University, Russia.}
\email{\tt m.alkousa@innopolis.ru}

\author[F.~S.~Stonyakin]{Fedor~S.~Stonyakin}
\address[F.~S.~Stonyakin]{Moscow Institute of Physics and Technology and V. I. Vernadsky Crimean Federal University, Russia.}
\email{{\tt fedyor@mail.ru}}


\begin{document}

\begin{abstract}
In this paper, we analyze the mirror descent algorithm for non-smooth optimization problems in which the objective function is relatively strongly convex, without relying on the standard Lipschitz continuity assumption commonly used in the literature. We provide convergence analyses for both exact and inexact subgradient information. Furthermore, through numerical experiments, we compare the derived bounds on the quality of the approximate solutions with existing estimates in the literature and demonstrate the effectiveness of the proposed results. 
\end{abstract}

\maketitle

\section{Introduction}
The first method for unconstrained minimization of a non-smooth convex function was proposed with a fixed step size in \cite{shor_book}. In subsequent years, several strategies for selecting step sizes were developed; see, for example, \cite{polyak1987introduction}. These include fixed step sizes, non-summable diminishing step sizes, square-summable but not summable step sizes \cite{Boyd2004Subgradient}, the Polyak step size \cite{polyak1987introduction}, and many others. For convex and strongly convex objective functions that are Lipschitz continuous, the optimal convergence rates of the subgradient method are well established; see \cite{Bubeck_book,Nesterov_book} for the convex case and \cite{Bach2012simpler} for the strongly convex case.

A major subsequent development in this area is related to the development of the mirror descent method, originally proposed in \cite{Nemirovskii1979efficient,Nemirovsky1983Complexity} and later analyzed in \cite{Beck2003Mirror}. Mirror descent can be viewed as a non-Euclidean generalization of the standard subgradient method. It has been successfully applied in a wide range of applications; see, for instance, \cite{applications_tomography_2001,article:Nazin_2011,article:Nazin_2014,Lai2026Capacity} and the references therein. Moreover, mirror descent is applicable to optimization problems in Banach spaces, where gradient descent methods are generally not suitable \cite{article:doan_2019}. Extensions of mirror descent to constrained optimization problems have been proposed and studied in, among others, \cite{Bayandina2018Mirror,article:beck_comirror_2010,Nemirovsky1983Complexity,Stonyakin2019some,Stonyakin2019Adaptive}.

In the literature on subgradient and mirror descent methods, when the objective function is convex or strongly convex, a common assumption is Lipschitz continuity of the objective function (see \eqref{gfdsdfg8521} and \eqref{lip_2} below). However, many convex functions encountered in practice fail to satisfy this assumption, even on compact convex sets. A simple example is the function $f(x) = - \sqrt{x}$ on the interval $[0,1]$. Furthermore, in most practical problems, the exact value of the Lipschitz constant is either unknown, difficult to estimate, or excessively large, which significantly deteriorates theoretical bounds on the quality of approximate solutions produced by first-order algorithms. These observations motivate the study of convergence properties of optimization algorithms for non-smooth functions without relying on Lipschitz continuity assumptions, leading to so-called Lipschitz-free methods.

Several attempts in this direction have been made. For example, in \cite{Ablaev2022Subgradient}, convergence guarantees for the projected subgradient method applied to strongly convex, non-Lipschitz functions were established (see Theorem 8 therein). In \cite{Xia2025Lipschitz}, a novel approach was proposed that extends the convergence analysis of the projected subgradient method to the non-Lipschitz convex setting. Subsequently, the results of \cite{alkousa2024optimal} and \cite{Xia2025Lipschitz} were further extended in \cite{Bowen2025Lipschitz}, where Lipschitz-free mirror descent algorithms for convex optimization problems were introduced.

In many large-scale optimization problems, exact (sub)gradients are not available, which limits the applicability of classical first-order methods. A standard example is stochastic optimization, where only noisy stochastic approximations of gradients can be computed. Another important class of problems involves deterministic optimization with deterministic errors in the evaluation of (sub)gradients and/or objective function values. Such errors often arise when subproblems within an algorithm cannot be solved exactly due to computational complexity. A notable example is PDE-constrained optimization \cite{baraldi2023proximal}. Inexact gradients also naturally occur in optimal control and inverse problems, where the computation of gradients requires solving ODEs or PDEs \cite{matyukhin2021convex}. For a broader discussion of inexact gradient models and related algorithms, we refer to \cite{devolder2013intermediate,kornilov2023intermediate,polyak1987introduction,stonyakin2021inexact,vasin2023accelerated} and the references therein. These applications further motivate the study of first-order methods with inexact information.

In this paper, we analyze the mirror descent algorithm for the class of relatively strongly convex functions without assuming Lipschitz continuity of the objective function. Our analysis is carried out in two settings: first, when exact subgradient information is available, and second, when only inexact subgradient information is accessible. In particular, we consider both absolute and relative inexactness models (see \eqref{inexact_grad}, \eqref{A_inexact_mirror}, and \eqref{B_inexact_mirror}).

The remainder of the paper is organized as follows. Section~\ref{sect_basics} reviews basic definitions and preliminaries related to mirror descent. Section~\ref{sec:LipFreeMirrorStrongly} is devoted to the analysis of the mirror descent algorithm without the Lipschitz continuity assumption in the case of exact subgradient information. In Section~\ref{alalysis_inexact}, we extend this analysis to settings with absolute and relative inexact subgradients. Finally, Section~\ref{sect_numerical} presents numerical experiments that illustrate the effectiveness of the proposed methods and theoretical results.

\section{Fundamentals }\label{sect_basics}

Let $(\mathbf{E},\|\cdot\|)$ be a normed finite-dimensional vector space, with an arbitrary norm $\|\cdot\|$, and $\mathbf{E}^*$ be the conjugate space of $\mathbf{E}$ with the following norm
$$
    \|y\|_{*}=\max\limits_{x \in \mathbf{E}}\{\langle y,x\rangle: \|x\|\leq1\},
$$
where $\langle y,x\rangle$ is the value of the continuous linear functional $y \in \mathbf{E}^*$ at $x \in \mathbf{E}$.

Let $Q \subset \mathbf{E}$ be a compact convex set, and 
$\psi: Q \longrightarrow \mathbb{R}$ be a proper closed differentiable and $1$-strongly convex (called prox-function or distance generating function). The corresponding Bregman divergence is defined as 
$$
    V_{\psi} (x, y) = \psi (x) - \psi (y) - \langle \nabla \psi (y), x - y \rangle, \quad \forall x , y \in Q. 
$$
For the Bregmann divergence, it holds the following inequality
\begin{equation}\label{eq_breg}
    V_{\psi} (x, y) \geq \frac{1}{2} \|y - x\|^2, \quad \forall x, y \in Q. 
\end{equation}

In what follows, we denote the subdifferential of $f$ at $x$ by $\partial f(x)$, and the subgradient of $f$ at any point $x$ by $\nabla f(x) \in \partial f(x)$.   Let $\operatorname{dom} (f)$ denote the domain of the function $f$, and $\operatorname{dom} (\partial f)$ denote the set of subdifferentiability points of $f$, i.e., 
$$
    \operatorname{dom} (\partial f) = \left\{ x \in  \mathbf{E} : \partial f(x) \ne \emptyset  \right\}. 
$$

The following identity, known as the three points identity, is essential in the analysis of the mirror descent method.

\begin{lemma}\label{three_points_lemma}(Three points identity) \cite{Chen1993}
Suppose that $\psi: \mathbf{E} \longrightarrow (- \infty, \infty]$ is a proper closed, convex, and differentiable function over $\operatorname{dom}(\partial \psi)$. Let $a, b \in \operatorname{dom}(\partial(\psi))$ and $c \in \operatorname{dom} (\psi)$. Then 
\begin{equation}\label{eq_three_points}
    \langle\nabla \psi(b)-\nabla \psi(a), c-a\rangle=V_{\psi}(c,a)+V_{\psi}(a, b)-V_{\psi}(c, b).
\end{equation}
\end{lemma}

\noindent
\textbf{Fenchel-Young inequality.} For any $a \in \mathbf{E}, b \in \mathbf{E}^*$, it holds the following inequality
\begin{equation}\label{Fenchel_Young_ineq}
    \left|  \langle a,b \rangle\right| \leq \frac{\|a\|^2}{2 \lambda} + \frac{\lambda\|b\|_*}{2}, \quad \forall \lambda > 0.
\end{equation}

We say that the function $f:Q \longrightarrow \mathbb{R}$ is a Lipschitz continuous function (non-smooth) with Lipschitz constant $M_f >0$, i.e., 
\begin{equation}\label{lipschitz_cond}
    |f(x) - f(y)| \leq M_f \|x - y\|, \quad \forall x, y \in Q. 
\end{equation}

The function $f: Q \longrightarrow \mathbb{R}$ (not necessarily differentiable) is called a $\mu$-strongly convex relative to $\psi(\cdot)$ (or  relatively $\mu$-strongly convex) in $Q$ if 
\begin{equation}\label{gfdsdfg8521}
    f(y) \geq f(x) + \langle \nabla f(x), y - x \rangle + \mu V_{\psi} (y, x) \quad \forall x, y \in Q.
\end{equation}
From \eqref{gfdsdfg8521}, we conclude the equivalent condition
\begin{equation}\label{lip_2}
    \|\nabla f(x)\|_* \leq M_f, \quad \text{where} \quad \nabla f(x) \in \partial f(x) \quad \forall x \in Q. 
\end{equation}

In this paper, we consider the following constrained optimization problem
\begin{equation}\label{main_constrained_prob}
    \min_{x \in Q} f(x),
\end{equation}
where $f$ is a relatively $\mu$-strongly convex (not necessarily differentiable) function. 


\section{Lipschitz-free mirror descent method for strongly convex functions}\label{sec:LipFreeMirrorStrongly}

In \cite{Bach2012simpler}, for the problem \eqref{main_constrained_prob}, when $\mathbf{E} = \mathbb{R}^n$, $f$ is $M_f$-Lipschitz and $\mu$-strongly convex, by using the subgradient method
\begin{equation}\label{subgrad_method}
    x_{k + 1} = \operatorname{Proj}_{Q} \left\{x_k - \frac{2}{\mu (k+1)} \nabla f(x_k)\right\},
\end{equation}
it was proved the following inequality
\begin{equation}\label{estim_subgrad_func}
    f(\widehat{x}) - f(x_*) \leq \frac{2M_f^2}{\mu(N+1)}, \quad \widehat{x} = \sum_{k = 1}^{N} \frac{2k}{N(N+1)} x_k,
\end{equation}
and
\begin{equation}\label{estim_subgrad_var}
    \|\widehat{x} - x_*\|_2^2 \leq \frac{4M_f^2}{\mu^2 (N+1)}.
\end{equation}
In this section, for problem \eqref{main_constrained_prob}, we will use the following Algorithm \ref{alg_mirror_descent}. 

\begin{algorithm}[!ht]
\caption{Mirror descent method.}\label{alg_mirror_descent}
\begin{algorithmic}[1]
    \REQUIRE step sizes $\{\gamma_k\}_{k \geq 1}$,  initial point $x_{1}  \in Q$, number of iterations $N \geq 1$.
    \FOR{$k= 1, 2, \ldots, N$}
    \STATE Calculate $\nabla f(x_k) \in \partial f(x_k)$,
    \STATE $x_{k+1} = \arg\min_{x \in Q} \left\{ \langle \nabla f(x_k), x   \rangle + \frac{1}{\gamma_k} V_{\psi}(x,x_k) \right\}. $ \label{line3_mirror}
    \ENDFOR
\end{algorithmic}
\end{algorithm}

For Algorithm \ref{alg_mirror_descent}, we have the following result.
\begin{theorem}\label{theo2_rate_projsubgrad_strongly}
Let $Q$ be a compact convex set, and $f: Q \longrightarrow \mathbb{R}$ be a relatively $\mu$-strongly convex. Then, for problem \eqref{main_constrained_prob}, by Algorithm \ref{alg_mirror_descent}, with $\gamma_k = \frac{2}{\mu (k+1)}$, it satisfies the following inequality
\begin{equation}\label{extime2_sub_strogly}
    f(\widehat{x}) - f(x_*) \leq  \frac{2}{\mu N(N+1)} \sum_{k = 1}^{N} \frac{k \|\nabla f(x_k)\|_*^2}{k+1}, 
\end{equation}
where $$\widehat{x} = \frac{1}{\sum_{k = 1}^{N} k} \sum_{k = 1}^{N} k x_k.$$

Moreover, if $f$ is $M_f$-Lipschitz, then to achieve the accuracy of $\varepsilon$ with respect to the function, i.e., $f(\widehat{x}) - f(x_*) \leq \varepsilon$, it is necessary $\mathcal{O} \left(\frac{M_f^2}{\mu\, \varepsilon}\right)$ iterations.    
\end{theorem} 

\begin{proof}
Let $F(x) :=  \langle \nabla f(x_k), x   \rangle + \frac{1}{\gamma_k} V_{\psi}(x,x_k)$, then from line \eqref{line3_mirror} in Algorithm \ref{alg_mirror_descent} we have $x_{k+1} =\arg\min\limits_{x \in Q} \tilde{f}(x) $. By the optimality condition, we get
$$
    \langle \nabla F(x_{k+1}), x - x_{k+1} \rangle \geq 0, \quad \forall x \in Q.
$$
Thus, 
$$
    \langle \gamma_k \nabla f(x_k) + \nabla \psi(x_{k+1})- \nabla \psi(x_k), x - x_{k+1} \rangle \geq 0, \quad \forall x \in Q.
$$
In particular, for $x = x_*$, we have
\begin{equation}\label{eq_yuio}
    \langle \nabla \psi(x_k) -\nabla \psi(x_{k+1}) - \gamma_k \nabla f(x_k), x_* - x_{k+1} \rangle \leq 0.
\end{equation}
Thus, we get
\begin{align}
    \gamma_k \langle \nabla f(x_k), x_k - x_* \rangle \nonumber
    & = \underbrace{\langle \nabla \psi(x_k) -\nabla \psi(x_{k+1}) - \gamma_k \nabla f(x_k), x_* - x_{k+1} \rangle}_{:=J_1} \nonumber 
    \\& \;\;\;\; + \underbrace{\langle \nabla \psi(x_{k+1}) -\nabla \psi(x_k), x_* - x_{k+1}  \rangle}_{: = J_2}  \nonumber
    \\& \;\;\;\; + \underbrace{\langle \gamma_k \nabla f(x_k), x_k - x_{k+1}  \rangle}_{: = J_3} . \label{uiop452lkj} 
\end{align} 

From \eqref{eq_yuio}, we find $J_1 \leq 0$. By Lemma \ref{three_points_lemma}, we have 
\begin{align*}
    J_2 & = - \langle \nabla \psi(x_k) -\nabla \psi(x_{k+1}), x_* - x_{k+1}  \rangle
    \\& = - \left( V_{\psi} (x_*, x_{k+1}) + V_{\psi} (x_{k+1}, x_k) - V_{\psi} (x_*, x_k) \right) 
    \\& = V_{\psi} (x_*, x_k) - V_{\psi} (x_*, x_{k+1}) - V_{\psi} (x_{k+1}, x_k). 
\end{align*} 

By the Fenchel-Young inequality \eqref{Fenchel_Young_ineq}, with $\lambda = 1 > 0 $, we find 
$$
    J_3 \leq \frac{\gamma_k^2}{2} \|\nabla f(x_k)\|_*^2 + \frac{1}{2} \|x_k - x_{k+1}\|^2. 
$$
Therefore, by combining with \eqref{uiop452lkj}, we get
\begin{align*}
    \gamma_k \langle \nabla f(x_k), x_k - x_* \rangle & \leq V_{\psi} (x_*, x_k) - V_{\psi} (x_*, x_{k+1}) - V_{\psi} (x_{k+1}, x_k)  
    \\& \quad + \frac{\gamma_k^2}{2} \|\nabla f(x_k)\|_*^2  + \frac{1}{2} \|x_k - x_{k+1}\|^2. 
\end{align*}
From \eqref{eq_breg}, we have 
$$
    V_{\psi} (x_{k+1}, x_k) \geq \frac{1}{2} \|x_{k+1} - x_k\|^2 .
$$
Thus, we get the following inequality
\begin{equation*}
    \gamma_k \langle \nabla f(x_k), x_k - x_* \rangle \leq  \frac{\gamma_k^2}{2} \|\nabla f(x_k)\|_*^2 +  V_{\psi} (x_*, x_k) - V_{\psi} (x_*, x_{k+1}). 
\end{equation*}
That is, 
\begin{equation}\label{eq_2kjhg}
     \langle \nabla f(x_k), x_k - x_* \rangle \leq  \frac{\gamma_k}{2} \|\nabla f(x_k)\|_*^2 + \frac{1}{\gamma_k} \left(  V_{\psi} (x_*, x_k) - V_{\psi} (x_*, x_{k+1}) \right). 
\end{equation}

Since $f$ is relatively $\mu$-strongly convex, from \eqref{gfdsdfg8521}, with $y: = x_*, x:= x_k$, we get
$$
    f(x_k) - f(x_*) + \mu V_{\psi}(x_*, x_k) \leq \langle \nabla f(x_k), x_k - x_*  \rangle. 
$$
Thus, from \eqref{eq_2kjhg} with $\gamma_k = \frac{2}{\mu (k+1)} \, \forall k \geq 1$, we get 
\begin{align*}
    & \quad \;\; 2k (f(x_k) - f(x_*)) + 2k \mu V_{\psi}(x_*, x_k) 
    \\& \leq \frac{2k}{\mu (k+1)} \|\nabla f(x_k)\|_*^2 + \mu k (k+1) \left(  V_{\psi} (x_*, x_k) - V_{\psi} (x_*, x_{k+1}) \right).
\end{align*}
By taking the summation of both sides of the last inequality from $k = 1$ to $k =N$ (for any $N\geq 1$), we find
\begin{align*}
    & \quad\; \sum_{k = 1}^{N} 2 k \left(f(x_k) - f(x)\right) + 2 \mu \sum_{k = 1}^{N} k V_{\psi}(x_*, x_k)
    \\& \leq \frac{2}{\mu} \sum_{k = 1}^{N} \frac{k}{k+1}\|\nabla f(x_k)\|_*^2 + \mu \sum_{k = 1}^{N} k (k+1) \left(  V_{\psi} (x_*, x_k) - V_{\psi} (x_*, x_{k+1}) \right)
    \\& = \frac{2}{\mu} \sum_{k = 1}^{N} \frac{k}{k+1}\|\nabla f(x_k)\|_*^2 + \mu \Big( 1 \cdot 2\left( V_{\psi} (x_*, x_1) - V_{\psi} (x_*, x_{2}) \right)   
    \\& \quad + 2 \cdot 3\left( V_{\psi} (x_*, x_2) - V_{\psi} (x_*, x_{3}) \right) + \ldots + (N-1)N\left( V_{\psi} (x_*, x_{N-1}) - V_{\psi} (x_*, x_{N}) \right) 
    \\& \quad + N (N+1)\left( V_{\psi} (x_*, x_N) - V_{\psi} (x_*, x_{N+1}) \right) \Big)
    \\& =\frac{2}{\mu} \sum_{k = 1}^{N} \frac{k}{k+1}\|\nabla f(x_k)\|_*^2 + \mu \Big( 2 V_{\psi} (x_*, x_{1}) + 4 V_{\psi} (x_*, x_{2}) + 6 V_{\psi} (x_*, x_{3}) + \ldots +
    \\& \qquad \qquad \qquad \qquad \qquad \qquad \quad \;\; + 2 N V_{\psi} (x_*, x_{N-1})  -  N(N+1) V_{\psi} (x_*, x_{N+1})\Big)
    \\& \leq \frac{2}{\mu} \sum_{k = 1}^{N} \frac{k}{k+1}\|\nabla f(x_k)\|_*^2 + 2\mu \left(V_{\psi} (x_*, x_{1}) + 2 V_{\psi} (x_*, x_{2}) + \ldots + N V_{\psi} (x_*, x_{N})\right)
    \\& = \frac{2}{\mu} \sum_{k = 1}^{N} \frac{k}{k+1}\|\nabla f(x_k)\|_*^2 +2 \mu \sum_{k = 1}^{N} k V_{\psi} (x_*, x_{k}). 
\end{align*}
Therefore, for any $N \geq 1$, we get 
$$
    \frac{1}{N(N+1)}\sum_{k = 1}^{N} 2 k \left(f(x_k) - f(x_*)\right) \leq \frac{2}{\mu N(N+1)} \sum_{k = 1}^{N} \frac{k}{k+1}\|\nabla f(x_k)\|_*^2. 
$$
This means 
$$
    \frac{2}{N(N+1)}\sum_{k = 1}^{N} k f(x_k)  -\frac{2}{N(N+1)}\sum_{k = 1}^{N}  k f(x_*)\leq \frac{2}{\mu N(N+1)} \sum_{k = 1}^{N} \frac{k}{k+1}\|\nabla f(x_k)\|_*^2.
$$
But, since $\sum_{k = 1}^{N} k = \frac{N(N+1)}{2}$, we get
$$
    \left(\frac{1}{\sum_{k = 1}^{N} k}\sum_{k = 1}^{N} k f(x_k)\right)  - f(x_*)\leq \frac{1}{\mu \sum_{k = 1}^{N} k } \sum_{k = 1}^{N} \frac{k}{k+1}\|\nabla f(x_k)\|_*^2.
$$
Because $f$ is convex, by Jensen's inequality, we find
$$
    f \left(\frac{1}{\sum_{k = 1}^{N} k}\sum_{k = 1}^{N} k x_k \right) \leq \frac{1}{\sum_{k = 1}^{N} k}\sum_{k = 1}^{N} k f(x_k). 
$$
Therefore, with $\widehat{x} : = \frac{1}{\sum_{k = 1}^{N} k} \sum_{k = 1}^{N} k x_k$, we get 
\begin{align*}
    f(\widehat{x}) - f(x_*) &\leq  \frac{1}{\mu \sum_{k = 1}^{N} k} \sum_{k = 1}^{N} \frac{k \|\nabla f(x_k)\|_*^2}{k+1}  \nonumber
    \\&= \frac{2}{\mu N (N+1) } \sum_{k = 1}^{N} \frac{k \|\nabla f(x_k)\|_*^2}{k+1}.
\end{align*}


Moreover, if $f$ is $M_f$-Lipschitz, then $\|\nabla f(x_k)\|_* \leq M_f \, \forall k \geq 1$. Thus, we have 
\begin{align}\label{classic_rate_bach}
    f(\widehat{x}) - f(x_*) & \leq \frac{2}{\mu N(N+1)} \sum_{k = 1}^{N} \frac{k M_f^2}{k+1} \nonumber
    \\& \leq \frac{2}{\mu N(N+1)}\sum_{k = 1}^{N} M_f^2 = \frac{2 M_f^2}{\mu (N+1)}. 
\end{align}

Now, let us set $\frac{2M_f^2}{\mu (N+1)} \leq \varepsilon$, then $N \geq \frac{2M_f^2}{ \mu \, \varepsilon} - 1$. This means that to achieve the accuracy of $\varepsilon$ with respect to the function, i.e., $f(\widehat{x}) - f(x_*) \leq \varepsilon$, it is necessary $\mathcal{O} \left(\frac{M_f^2}{\mu\, \varepsilon}\right)$ iterations.
\end{proof}

\begin{remark}
Note that the estimate in \eqref{classic_rate_bach} is the classical one, and the convergence rate $\mathcal{O}(1/\varepsilon)$ is optimal \cite{Bach2012simpler}. In our work, this rate is recovered as a special case of a more general analysis that applies to Lipschitz-free objective functions.
\end{remark}

\begin{remark}\label{remark_cov_variable}
From \eqref{gfdsdfg8521} and the optimality condition, we have 
\begin{equation*}
    f(\widehat{x}) - f(x_*) \geq  \mu V_{\psi} (\widehat{x}, x_*) \quad \text{where} \quad \widehat{x} : = \frac{1}{\sum_{k = 1}^{N} k} \sum_{k = 1}^{N} k x_k.
\end{equation*}
Thus, from \eqref{extime2_sub_strogly}  we find
\[
    V_{\psi} (\widehat{x}, x_*) \leq \frac{2}{\mu^2 N (N+1) } \sum_{k = 1}^{N} \frac{k \|\nabla f(x_k)\|_*^2}{k+1},
\]
or
\[
    \|\widehat{x} - x_*\|^2 \leq \frac{4}{\mu^2 N (N+1) } \sum_{k = 1}^{N} \frac{k \|\nabla f(x_k)\|_*^2}{k+1}.
\]
That is,  
\begin{equation}\label{estimate_variable_new}
    \|\widehat{x} - x_*\| \leq \frac{2}{ \mu\sqrt{ N (N+1)} } \sqrt{\sum_{k = 1}^{N} \frac{k \|\nabla f(x_k)\|_*^2}{k+1}}.
\end{equation}

Moreover, as a special case, if $f$ is $M_f$-Lipschitz, then we get the following inequality
\begin{equation}\label{estimate_variable_old}
    \|\widehat{x} - x_*\| \leq \frac{2 M_f}{ \mu \sqrt{N+1}},
\end{equation}
which is the same inequality in \cite{Bach2012simpler}
\end{remark}

\section{Lipschitz-free mirror descent method for strongly convex functions with inexact information}\label{alalysis_inexact}

In Section~\ref{sec:LipFreeMirrorStrongly}, we derived an estimate for the quality of an approximate solution to problem~\eqref{main_constrained_prob} obtained by the Mirror Descent Algorithm~\ref{alg_mirror_descent}, assuming access to an exact subgradient at each iteration. We now relax this assumption and consider the more realistic setting in which only an inexact subgradient $\widetilde{\nabla} f$ is available. In particular, we assume that the inexact subgradient satisfies either a relative or an absolute inexactness condition. That is, we have access to $\widetilde{\nabla} f$ such that
\begin{equation}\label{inexact_grad}
    \|\widetilde{\nabla} f(x) - \nabla f(x)\|_* \leq A. \quad \Longrightarrow \quad  \|\widetilde{\nabla} f(x)\|_* \leq B,
\end{equation}
where, 
\begin{equation}\label{A_inexact_mirror}
    A = \begin{cases}
    \alpha \|\nabla f(x)\|_*, & \text{for the relative inexactness} \, (\alpha \in [0,1));\\
    \Delta, & \text{for the absolute inexactness} \, (\Delta \geq 0),
    \end{cases}
\end{equation}
and 
\begin{equation}\label{B_inexact_mirror}
    B = \begin{cases}
    (1+\alpha) \|\nabla f(x)\|_*, & \text{for the relative inexactness;}\\
    \Delta +  \|\nabla f(x)\|_*, & \text{for the absolute inexactness.}
    \end{cases}
\end{equation}

For Algorithm \ref{alg_mirror_descent}, with inexact subgradient, we have the following result. 

\begin{theorem}\label{theo_free_mirror_strongly_inexact}
Let $Q \subset$ be a compact convex set, and $f: Q \longrightarrow \mathbb{R}$ be a relatively $\mu$-strongly convex. Then, for problem \eqref{main_constrained_prob}, by Algorithm \ref{alg_mirror_descent} (with an inexact subgradient), with $\gamma_k = \frac{4}{\mu (k+1)}$, we have
\begin{enumerate}
\item with relative inexact subgradient $\widetilde{\nabla}f$, it satisfies the following
\begin{equation}\label{extime_relative_mirror_free}
    f(\widehat{x}) - f(x_*) \leq  \frac{4(1+ \alpha)^2}{\mu N(N+1)} \sum_{k = 1}^{N} \|\nabla f(x_k)\|_*^2 + \frac{2 \alpha^2}{\mu N (N+1)} \sum_{k = 1}^{N} k \|\nabla f(x_k)\|_*^2,
\end{equation}  

\item with absolute inexact subgradient $\widetilde{\nabla}f$, it satisfies the following     
\begin{equation}\label{extime_absol_mirror_free}
    f(\widehat{x}) - f(x_*) \leq  \frac{4}{\mu N(N+1)} \sum_{k = 1}^{N} \left( \Delta + \|\nabla f(x_k)\|_*\right)^2 + \frac{\Delta^2}{\mu}, 
\end{equation}
\end{enumerate}
where $\widehat{x} = \frac{2}{N(N+1)} \sum_{k = 1}^{N} k x_k$.  
\end{theorem} 
\begin{proof}
Similarly, as in the proof of Theorem \ref{theo2_rate_projsubgrad_strongly}, by using an inexact subgradient $\widetilde{\nabla} f$, we get an analogue of \eqref{eq_2kjhg}
\begin{equation}\label{eq_squareww}
    \langle \widetilde{\nabla} f(x_k), x_k - x_* \rangle \leq  \frac{\gamma_k}{2} \|\widetilde{\nabla} f(x_k)\|_*^2 + \frac{1}{\gamma_k} \left(  V_{\psi} (x_*, x_k) - V_{\psi} (x_*, x_{k+1}) \right).
\end{equation}

Because $f$ is relatively $\mu$-strongly convex (relative to $\psi(\cdot)$), we have 
\begin{align*}
    & \qquad \;\; f(x_k) - f(x_*)  
    \\& \;\, \leq \; \; \langle \nabla f(x_k), x_k - x_* \rangle  - \mu V_{\psi} (x_*, x_k)
    \\& \;\, = \; \; \langle \widetilde{\nabla} f(x_k), x_k - x_* \rangle + \langle \nabla f(x_k) - \widetilde{\nabla} f(x_k) , x_k - x_* \rangle  - \mu V_{\psi} (x_*, x_k) 
    \\& \;\, \leq \; \; \langle \widetilde{\nabla} f(x_k) , x_k - x_* \rangle + A\|x_k - x_*\| - \mu V_{\psi} (x_*, x_k) 
    \\& \; \, \stackrel{\eqref{eq_squareww}}{\leq} \frac{\gamma_k}{2} \|\widetilde{\nabla}f(x_k)\|_*^2 + \frac{1}{\gamma_k} V_{\psi} (x_*, x_{k}) - \frac{1}{\gamma_k} V_{\psi} (x_*, x_{k+1})  - \mu V_{\psi} (x_*, x_k) 
    \\& \qquad \;\;  + A   \|x_k - x_*\|
    \\&  \; \, \leq \; \,  \frac{\gamma_k}{2} \|\widetilde{\nabla}f(x_k)\|_*^2+ \frac{1}{\gamma_k} V_{\psi} (x_*, x_{k}) - \frac{1}{\gamma_k} V_{\psi} (x_*, x_{k+1})  - \mu V_{\psi} (x_*, x_k) 
    \\& \qquad \;\; + \frac{A^2}{\mu} +  \frac{\mu}{4}   \|x_k - x_*\|^2
    \\& \; \, \stackrel{\text{(b)}}{\leq} \frac{\gamma_k}{2} \|\widetilde{\nabla}f(x_k)\|_*^2 + \left(\frac{1}{\gamma_k} - \frac{\mu}{2}\right)V_{\psi} (x_*, x_k)- \frac{1}{\gamma_k} V_{\psi} (x_*, x_{k+1}) + \frac{A^2}{\mu}. 
\end{align*}
Where, in (b) we used the inequality $V_{\psi} (x_*, x_k) \geq \frac{1}{2}\|x_k - x_*\|^2$. 

Since $\gamma_k = \frac{4}{\mu (k+1)}$, then by multiplying both sides of the last inequality by $2 k \geq 2 \, \forall k \geq 1$, taking the summation from $k = 1$ to $k = N$, we get
\begin{align*}
    & \quad 2 \sum_{k = 1}^{N} k f(x_k) - N (N+1) f(x_*) 
    \\& \leq \frac{4}{\mu}\sum_{k = 1}^{N} \frac{k}{k+1} \|\widetilde{\nabla}f(x_k)\|_*^2 + \frac{\mu}{2} \sum_{k = 1}^{N} \left(k(k-1) V_{\psi} (x_*, x_{k}) - k(k+1)V_{\psi} (x_*, x_{k+1}) \right) 
    \\& \quad + \frac{2}{\mu} \sum_{k = 1}^{N} kA^2
    \\& \leq   \frac{4}{\mu} \sum_{k = 1}^{N}\frac{k}{k+1} \|\widetilde{\nabla}f(x_k)\|_*^2+ \frac{2}{\mu} \sum_{k = 1}^{N} kA^2.
\end{align*}
Thus, 
\begin{align*}
    \frac{2}{N (N+1)}\sum_{k = 1}^{N} k f(x_k) -  f(x_*) & \leq \frac{4}{\mu N (N+1)} \sum_{k = 1}^{N} \frac{k}{k+1} \|\widetilde{\nabla} f(x_k)\|_*^2
    \\& \quad +  \frac{2}{\mu N (N+1)} \sum_{k = 1}^{N} k A^2.
\end{align*}

Because $f$ is a convex function, with $\widehat{x} = \frac{2}{N(N+1)} \sum_{k = 1}^{N} k x_k$, we get  
\[
    f(\widehat{x}) - f(x_*) \leq \frac{4}{\mu N (N+1)} \sum_{k = 1}^{N} \frac{k}{k+1} \|\widetilde{\nabla} f(x_k)\|_*^2 + \frac{2}{\mu N (N+1)} \sum_{k = 1}^{N} k A^2. 
\]

\begin{enumerate}
\item [\textbf{(1)}] For the relative inexactness in subgradient, from \eqref{A_inexact_mirror}, and \eqref{B_inexact_mirror}, we get 
\begin{align*}
    f(\widehat{x}) - f(x_*) & \leq \frac{4(1+ \alpha)^2}{\mu N (N+1)} \sum_{k = 1}^{N} \frac{k}{k+1} \|\nabla f(x_k)\|_*^2 
    \\& \quad + \frac{2 \alpha^2}{\mu N (N+1)} \sum_{k = 1}^{N} k \|\nabla f(x_k)\|_*^2
    \\& \leq \frac{4(1+ \alpha)^2}{\mu N (N+1)} \sum_{k = 1}^{N}  \|\nabla f(x_k)\|_*^2 + \frac{2 \alpha^2}{\mu N (N+1)} \sum_{k = 1}^{N} k \|\nabla f(x_k)\|_*^2.
\end{align*}
Which is the desired inequality \eqref{extime_relative_mirror_free}.

\item [\textbf{(2)}] For the absolute inexactness in subgradient, from \eqref{A_inexact_mirror}, and \eqref{B_inexact_mirror}, we get
\begin{align*}
    f(\widehat{x}) - f(x_*) & \leq \frac{4}{\mu N (N+1)} \sum_{k = 1}^{N} \frac{k}{k+1} \left(\Delta + \|\nabla f(x_k)\|_*\right)^2 
    \\& \quad + \frac{2}{\mu N (N+1)} \sum_{k = 1}^{N} k \Delta^2
    \\& \leq \frac{4}{\mu N (N+1)} \sum_{k = 1}^{N} \left(\Delta + \|\nabla f(x_k)\|_*\right)^2 + \frac{\Delta^2}{\mu}. 
\end{align*}
Which is the desired inequality \eqref{extime_absol_mirror_free}.
\end{enumerate}
\end{proof}

\begin{remark}
Similarly to Remark \ref{remark_cov_variable}, from \eqref{extime_relative_mirror_free} we find
\begin{equation*}
    \|\widehat{x} - x_*\|^2 \leq  \frac{8(1+ \alpha)^2}{\mu^2 N(N+1)} \sum_{k = 1}^{N} \|\nabla f(x_k)\|_*^2 + \frac{4 \alpha^2}{\mu^2 N (N+1)} \sum_{k = 1}^{N} k \|\nabla f(x_k)\|_*^2.
\end{equation*}
Therefore, we get
\begin{equation}\label{extime_var_relative}
    \|\widehat{x} - x_*\| \leq  \frac{2\sqrt{2}(1+ \alpha)}{\mu \sqrt{N(N+1)}} \sqrt{\sum_{k = 1}^{N} \|\nabla f(x_k)\|_*^2} + \frac{2 \alpha}{\mu\sqrt{ N (N+1)}} \sqrt{\sum_{k = 1}^{N} k \|\nabla f(x_k)\|_*^2}.
\end{equation}


Also, from \eqref{extime_absol_mirror_free} we find 
\begin{equation*}
    \|\widehat{x} - x_*\|^2 \leq  \frac{8}{\mu^2 N(N+1)} \sum_{k = 1}^{N} \left( \Delta + \|\nabla f(x_k)\|_*\right)^2 + \frac{2 \Delta^2}{\mu^2}. 
\end{equation*}
Therefore, we get 
\begin{equation}\label{extime_var_absol}
    \|\widehat{x} - x_*\| \leq  \frac{2\sqrt{2}}{\mu\sqrt{ N(N+1)} } \sqrt{\sum_{k = 1}^{N} \left( \Delta + \|\nabla f(x_k)\|_*\right)^2} + \frac{\sqrt{2}\Delta}{\mu}. 
\end{equation}
\end{remark}

\section{Numerical experiments}\label{sect_numerical}

To illustrate the advantages of the proposed analysis of the Mirror Descent algorithm for relatively strongly convex optimization problems without the Lipschitz continuity assumption, we consider a special setting in which the underlying space is $\mathbb{R}^n$ equipped with the Euclidean distance. In this setting, we conduct numerical experiments to compare the classical estimates of the quality of approximate solutions obtained by the subgradient method~\eqref{subgrad_method}. In particular, we compare estimates based on functional values~\eqref{estim_subgrad_func} and on the iterates~\eqref{estim_subgrad_var} with the corresponding bounds derived in our analysis, namely~\eqref{extime2_sub_strogly} and~\eqref{estimate_variable_new}, assuming access to exact subgradient information. We also investigate the behavior of the estimates~\eqref{extime_relative_mirror_free}, \eqref{extime_absol_mirror_free}, \eqref{extime_var_relative}, and~\eqref{extime_var_absol} as functions of the number of iterations for different values of the inexactness parameters $\alpha$ and $\Delta$.

The experiments are conducted over the feasible set given by the Euclidean ball centered at the origin with radius $R$, namely, 
$$
    Q = \left\{ x \in \mathbb{R}^n : \|x\|_2 \leq R \right\},
$$
with the initial point chosen as $x_1 = \left( \frac{R}{\sqrt{n}}, \ldots, \frac{R}{\sqrt{n}} \right) \in Q.$ We consider two illustrative examples.
\begin{example}\label{ex_1}
The objective function in this example is
$$
    f(x) = \|x\|_2 + 2 \gamma \|x\|_2^2.
$$
This function is strongly convex with $\mu = 2 \gamma$, and Lipschitz continuous with $M_f = 1 + 2 \gamma R$. 
\end{example}

\begin{example}\label{ex_2}
Let $\{A_1, \ldots, A_2\}$ be a set of $m$ points in $\mathbb{R}^n$. The objective function in this example is 
$$
    f(x) = \max_{1 \leq i \leq m} \{\|x - A_i\|_2^2\}.
$$
This function is $2$-strongly convex and Lipschitz continuous with $M_f = 2\left(R + \max_{1 \leq i \leq m}\|A_i\|_2\right)$. 
\end{example}

The results are presented in Figures~\ref{fig1}--\ref{fig4}. From these figures, we observe that the theoretical estimates of the quality of the solution derived in this work, both in terms of functional values and iterates, are sharper than the classical, and they improve upon the corresponding estimates concluded in \cite{Bach2012simpler}.

\begin{figure}[htp]
    \minipage{0.50\textwidth}
    \includegraphics[width=\linewidth]{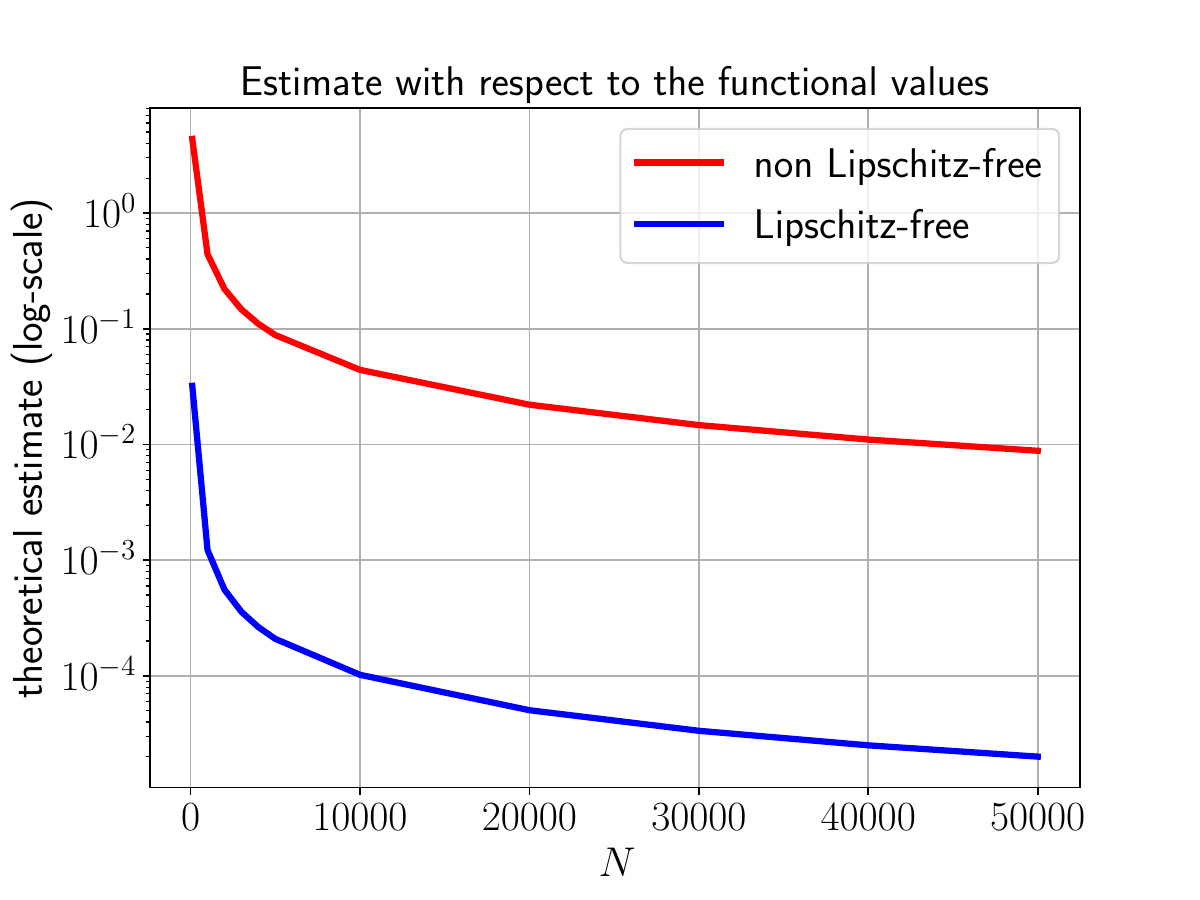}
    \endminipage
    \minipage{0.50\textwidth}
    \includegraphics[width=\linewidth]{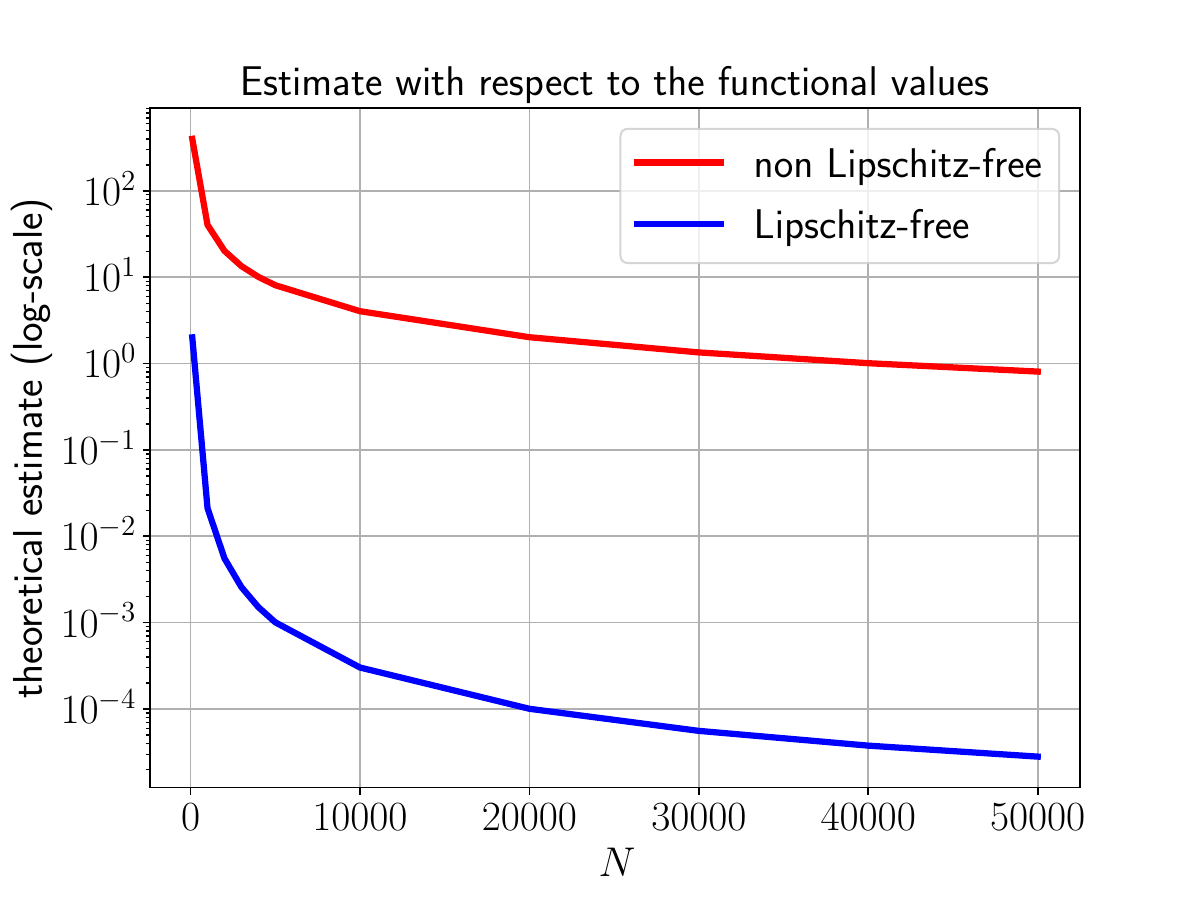}
    \endminipage
    \caption{Results of comparison of theoretical estimates with respect to the functional values \eqref{estim_subgrad_func} and \eqref{extime2_sub_strogly}, for Example \ref{ex_1}, with $n=1000, R = 10$ (left). And with $n = 1000, R = 100$ (right).}
    \label{fig1}
\end{figure} 

\begin{figure}[htp]
    \minipage{0.50\textwidth}
    \includegraphics[width=\linewidth]{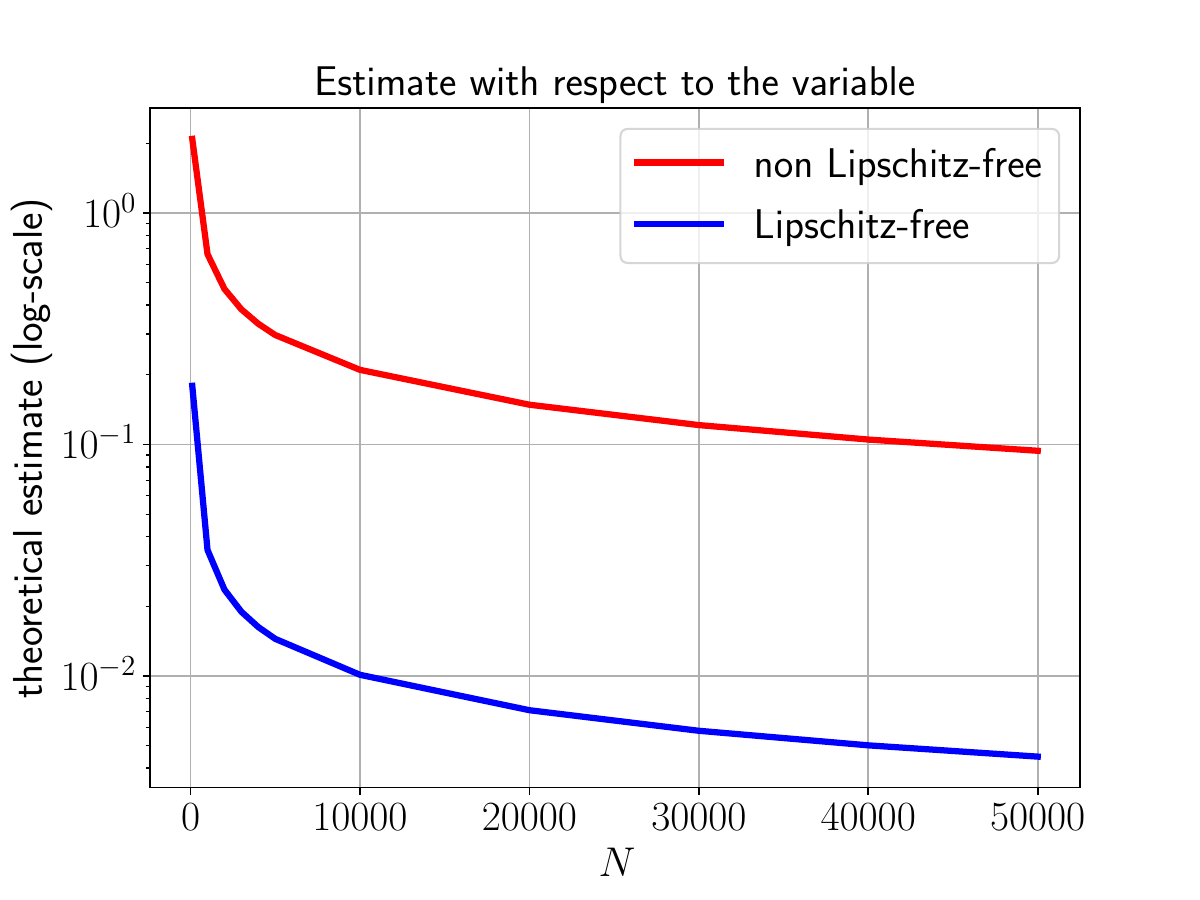}
    \endminipage
    \minipage{0.50\textwidth}
    \includegraphics[width=\linewidth]{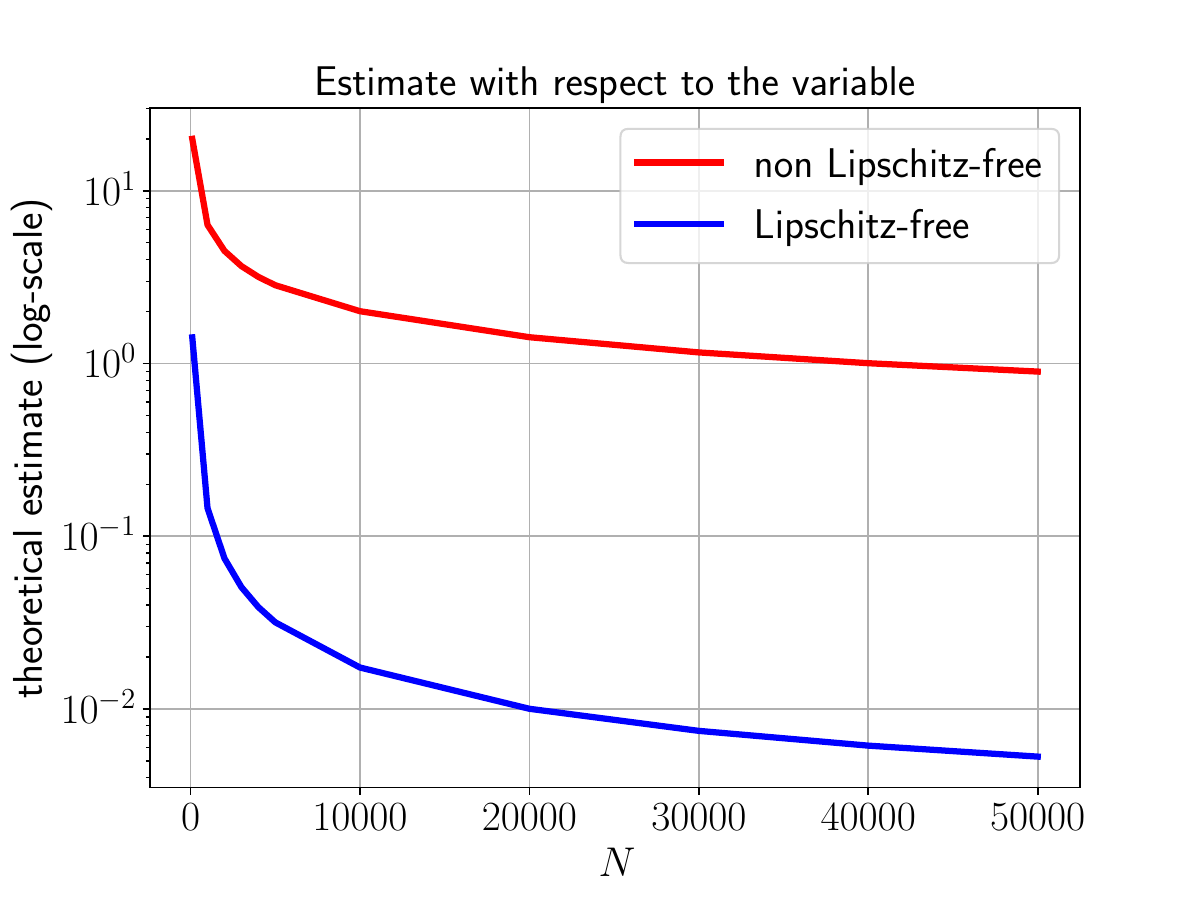}
    \endminipage
    \caption{Results of comparison of theoretical estimates with respect to the variable \eqref{estim_subgrad_var} and \eqref{estimate_variable_new}, for Example \ref{ex_1}, with $n=1000, R = 10$ (left). And with $n = 1000, R = 100$ (right).}
    \label{fig2}
\end{figure}

\begin{figure}[htp]
    \minipage{0.50\textwidth}
    \includegraphics[width=\linewidth]{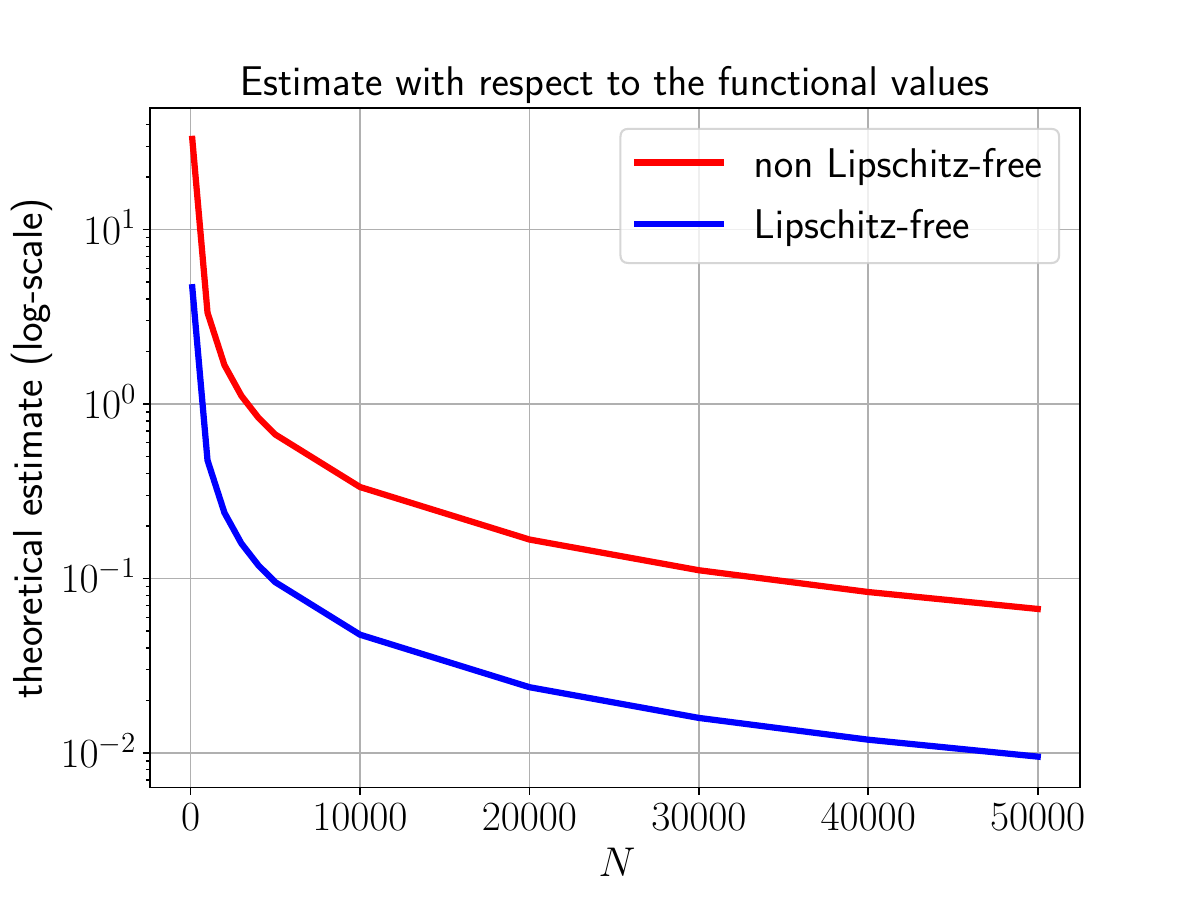}
    \endminipage
    \minipage{0.50\textwidth}
    \includegraphics[width=\linewidth]{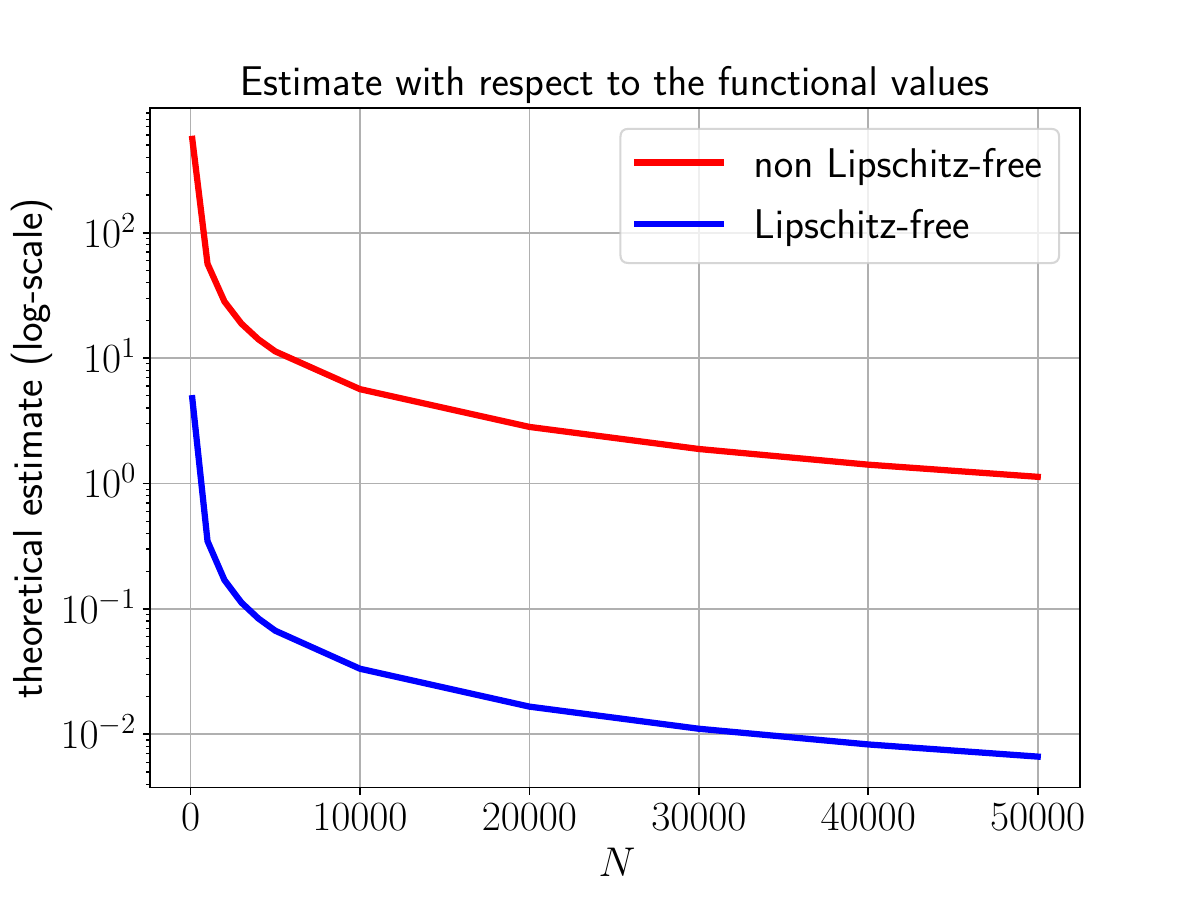}
    \endminipage
    \caption{Results of comparison of theoretical estimates with respect to the functional values \eqref{estim_subgrad_func} and \eqref{extime2_sub_strogly}, for Example \ref{ex_2}, with $m = 50, n=1000, R = 10$ (left). And with $m = 50, n = 1000, R = 100$ (right).}
    \label{fig3}
\end{figure} 

\begin{figure}[htp]
    \minipage{0.50\textwidth}
    \includegraphics[width=\linewidth]{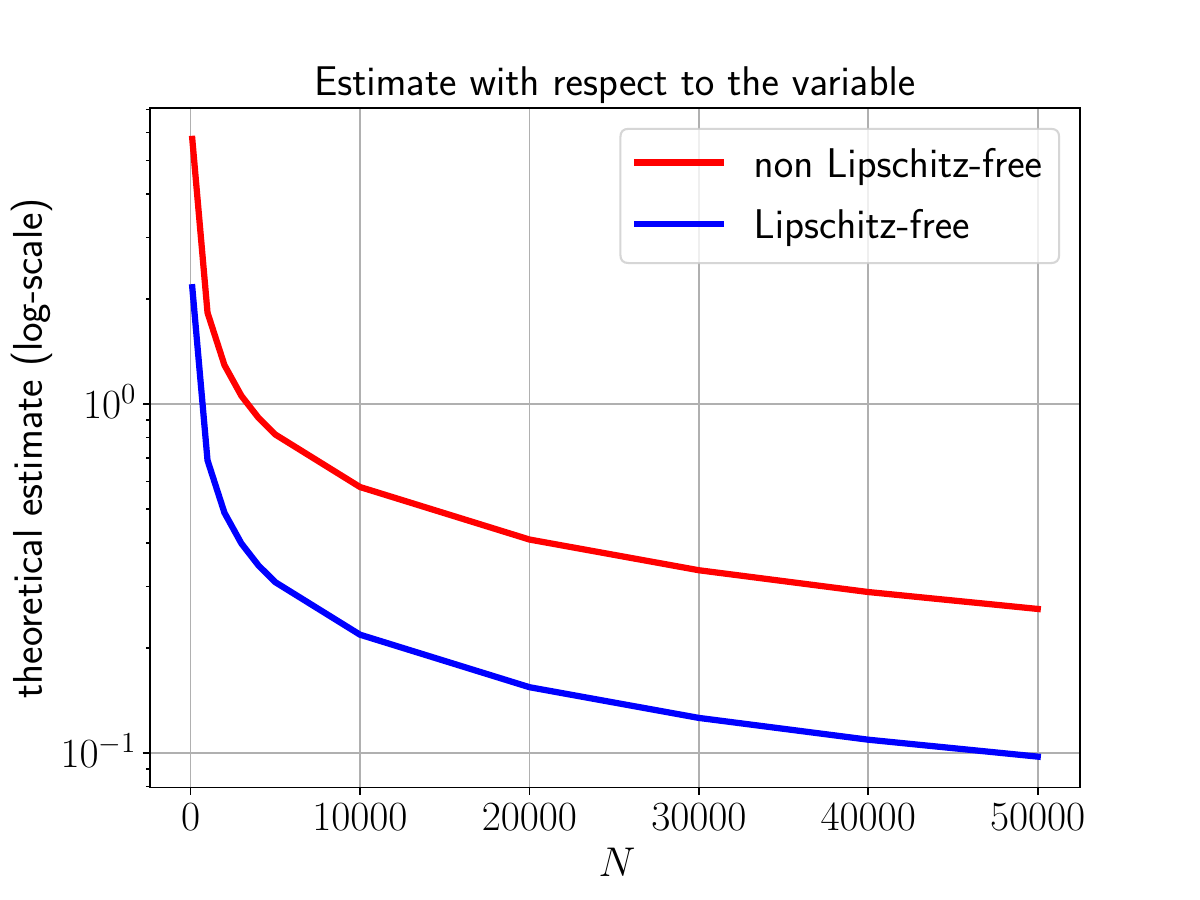}
    \endminipage
    \minipage{0.50\textwidth}
    \includegraphics[width=\linewidth]{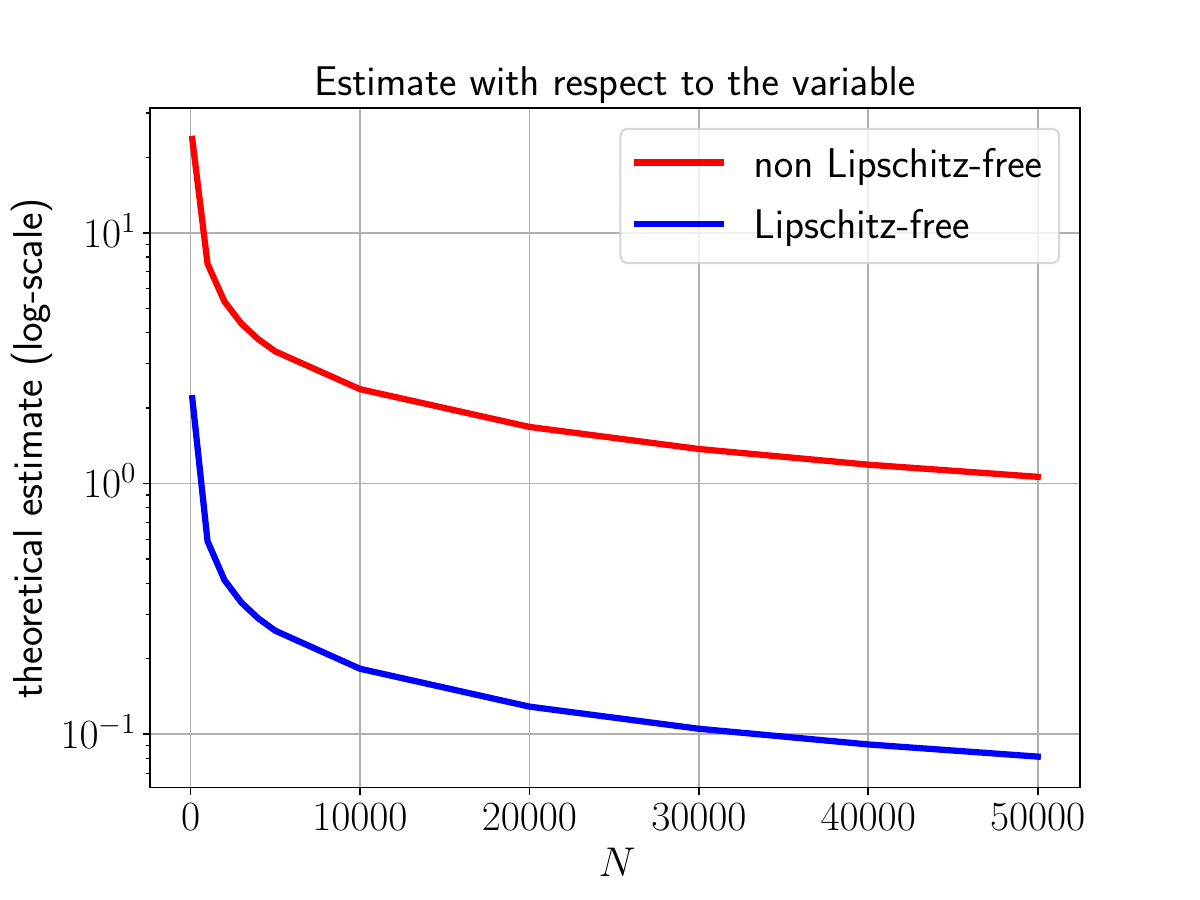}
    \endminipage
    \caption{Results of comparison of theoretical estimates with respect to the variable \eqref{estim_subgrad_var} and \eqref{estimate_variable_new}, for Example \ref{ex_2}, with $n=1000, R = 10$ (left). And with $n = 1000, R = 100$ (right).}
    \label{fig4}
\end{figure} 

Next, we illustrate the behavior of the estimates~\eqref{extime_relative_mirror_free}, \eqref{extime_absol_mirror_free}, \eqref{extime_var_relative}, and~\eqref{extime_var_absol} as functions of the number of iterations, for different values of the inexactness parameters $\alpha$ and $\Delta$, for Examples~\ref{ex_1} and~\ref{ex_2}.

\begin{figure}[htp]
    \minipage{0.50\textwidth}
    \includegraphics[width=\linewidth]{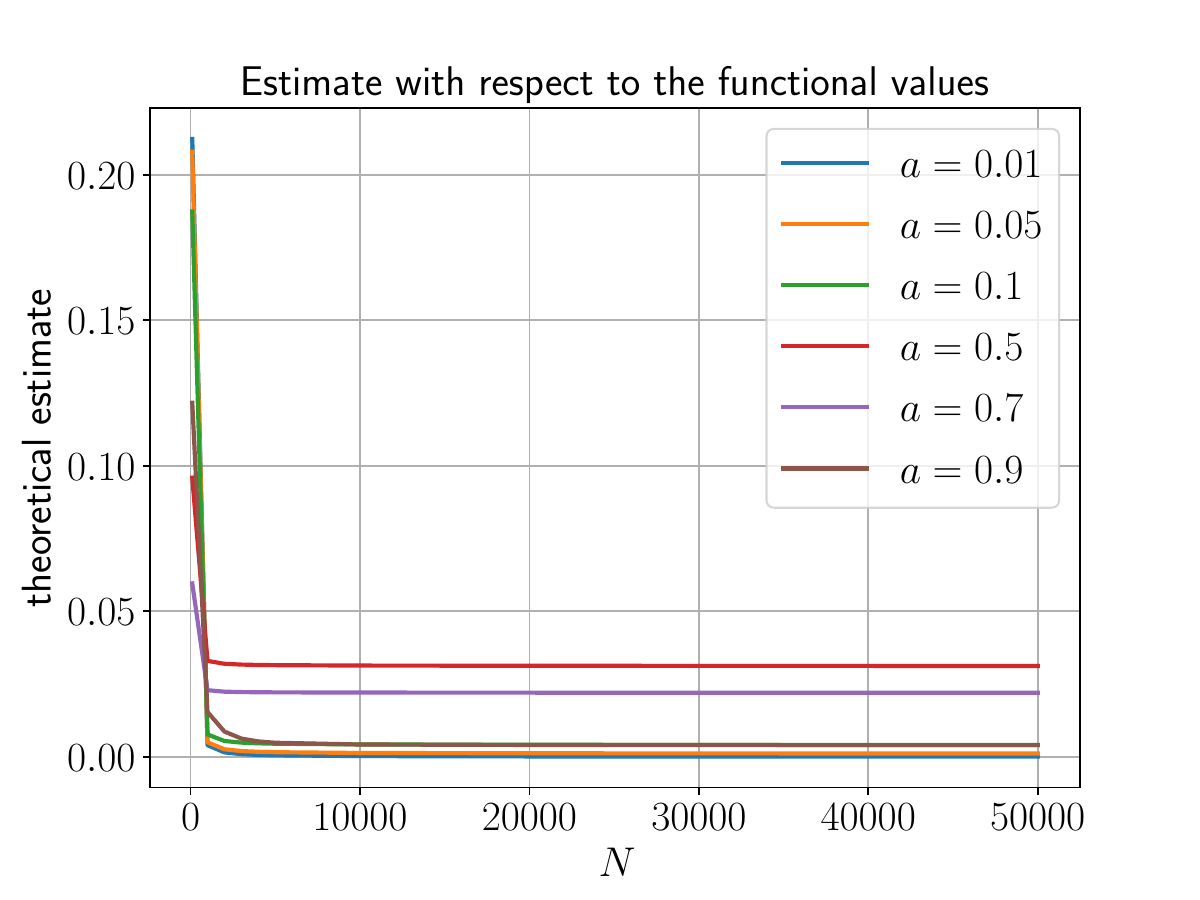}
    \endminipage
    \minipage{0.50\textwidth}
    \includegraphics[width=\linewidth]{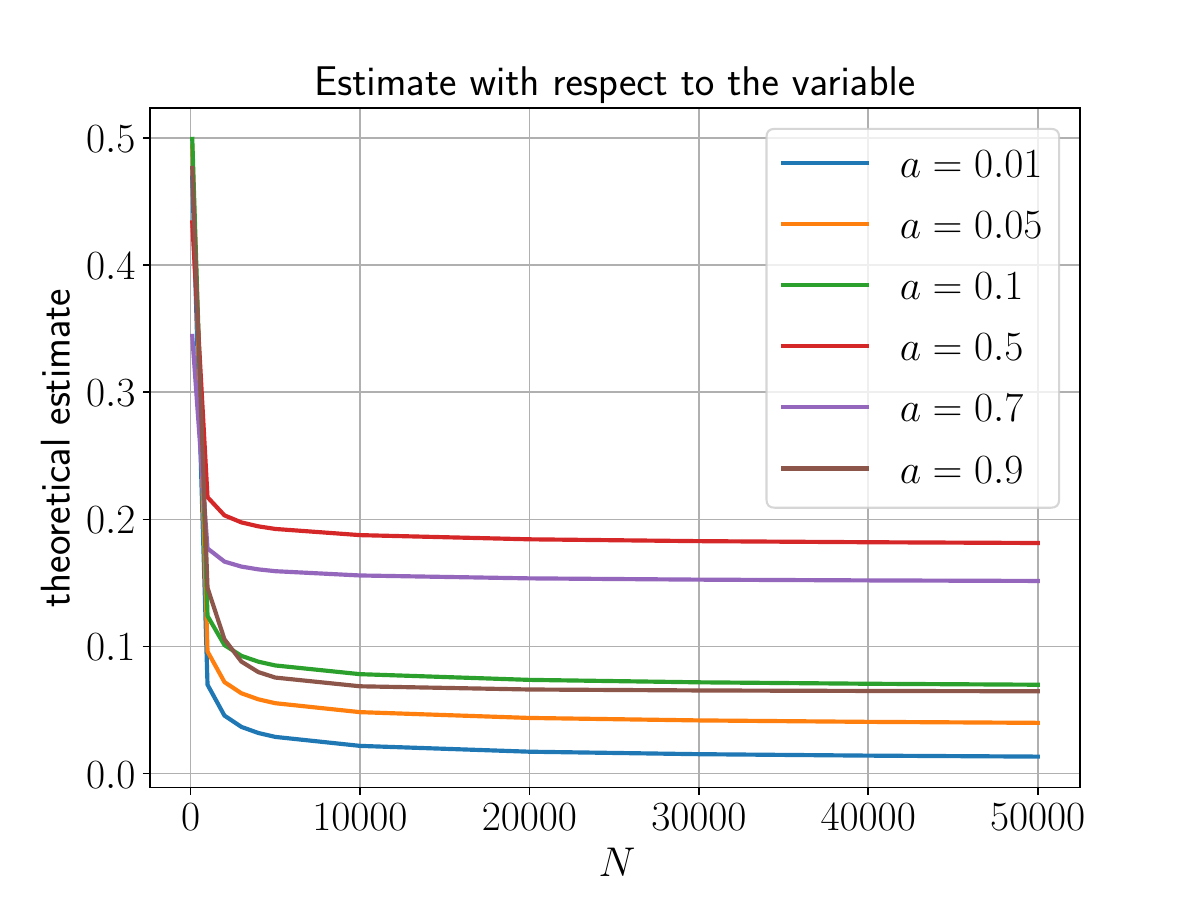}
    \endminipage
    \caption{Theoretical estimates with respect to the functional values \eqref{extime_relative_mirror_free} and variable \eqref{extime_var_relative}, for Example \ref{ex_1}, with $n=1000,$ and $ R = 10$.}
    \label{fig5}
\end{figure} 

\begin{figure}[htp]
    \minipage{0.50\textwidth}
    \includegraphics[width=\linewidth]{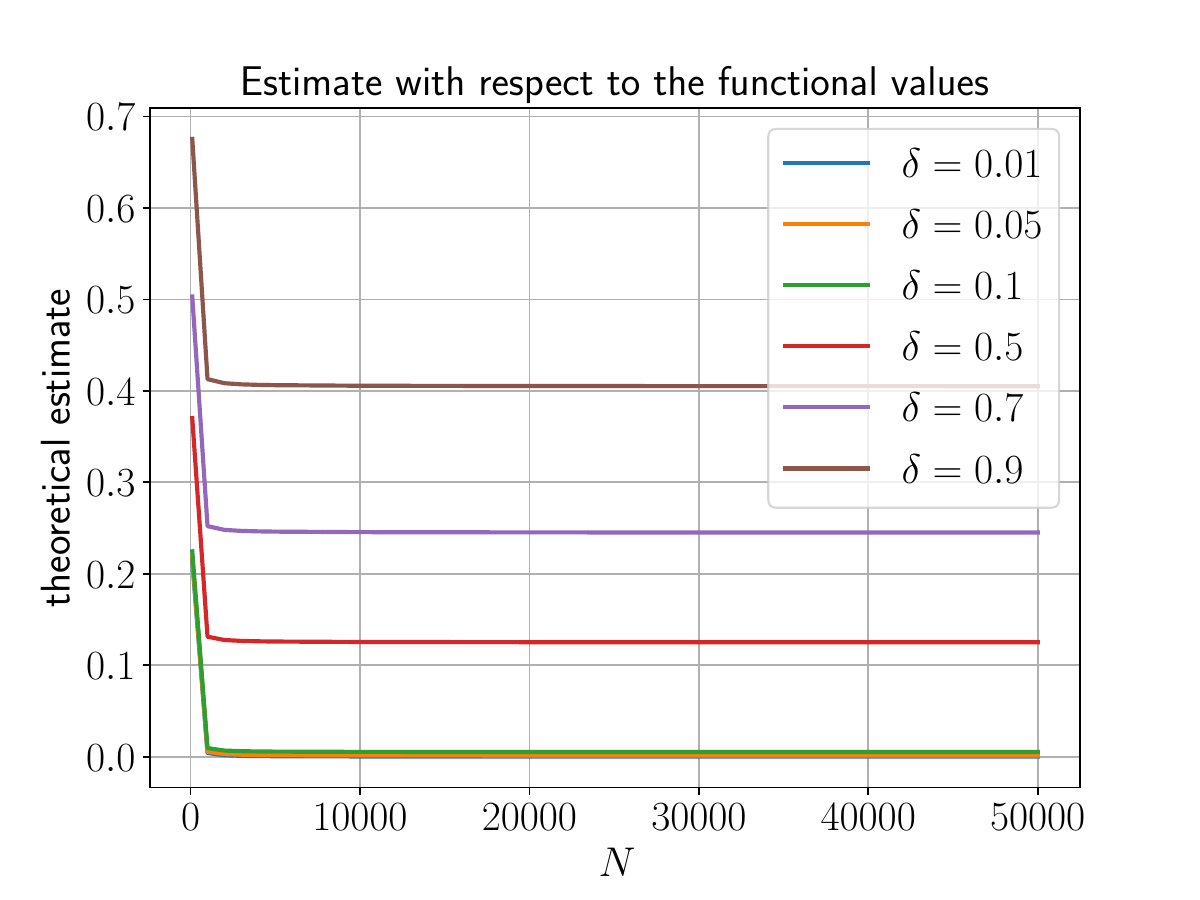}
    \endminipage
    \minipage{0.50\textwidth}
    \includegraphics[width=\linewidth]{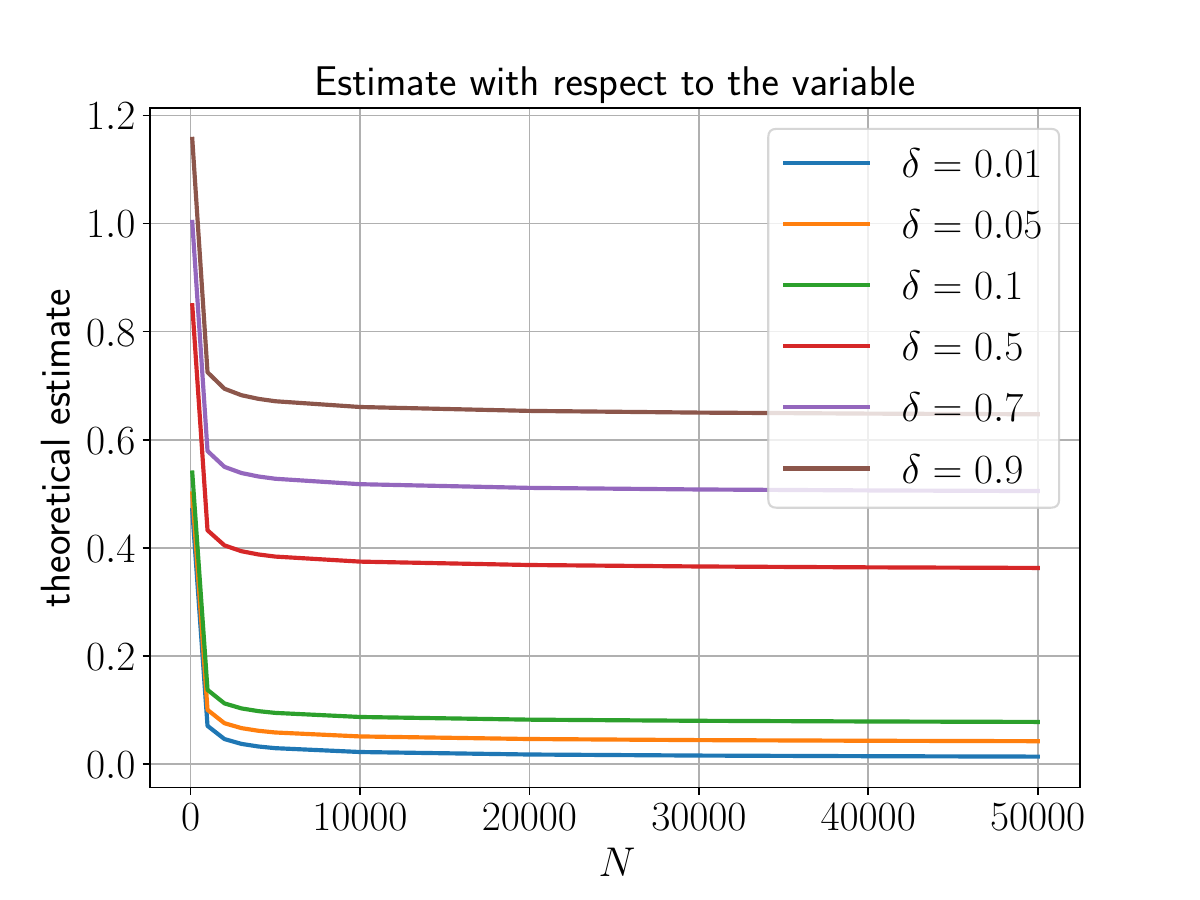}
    \endminipage
    \caption{Theoretical estimates with respect to the functional values \eqref{extime_absol_mirror_free} and variable \eqref{extime_var_absol}, for Example \ref{ex_1}, with $n=1000,$ and $ R = 10$.}
    \label{fig6}
\end{figure} 


\section{Conclusions}

In this paper, we analyzed the mirror descent algorithm for non-smooth optimization problems with relatively strongly convex objective functions without assuming Lipschitz continuity. A Lipschitz-free convergence analysis was developed for exact and inexact subgradient information, including relative and absolute inexactness models. The classical optimal convergence rate $\mathcal{O}(1/\varepsilon)$, in the case of exact subgradient information, was recovered as a special case of the performed  analysis.

To illustrate the practical relevance of the achieved results, we conducted numerical experiments in the Euclidean setting. The experiments demonstrated that the derived theoretical bounds are sharper than the classical estimates available in the literature and remain robust under inexact gradient information. In particular, the results highlight the advantages of the proposed Lipschitz-free mirror descent analysis in scenarios where Lipschitz constants are unknown, difficult to estimate, or excessively large.



\begin{thebibliography}{99}
\bibitem{alkousa2024optimal}
Alkousa, M., Stonyakin, F., Abdo, A., Alcheikh, M. (2024). Optimal Convergence Rate for Mirror Descent Methods with Special Time-Varying Step Sizes Rules. In: Eremeev, A., Khachay, M., Kochetov, Y., Mazalov, V., Pardalos, P. (eds) Mathematical Optimization Theory and Operations Research: Recent Trends. MOTOR 2024. Communications in Computer and Information Science, vol 2239. Springer, Cham.

\bibitem{Ablaev2022Subgradient}
Ablaev, S. S., Makarenko, D. V., Stonyakin, F. S., Alkousa, M. S., \& Baran, I. V. (2022). Subgradient methods for non-smooth optimization problems with some relaxation of sharp minimum. arXiv preprint arXiv:2212.06055. 

\bibitem{baraldi2023proximal} 
Baraldi, R. J., Kouri, D. P.: A proximal trust-region method for nonsmooth optimization with inexact function and gradient evaluations. Mathematical Programming, 201(1), 559-598 (2023)

\bibitem{Bayandina2018Mirror}
Bayandina, A., Dvurechensky, P., Gasnikov, A., Stonyakin, F., \& Titov, A. (2018). Mirror descent and convex optimization problems with non-smooth inequality constraints. In Large-scale and distributed optimization (pp. 181-213). Cham: Springer International Publishing.

\bibitem{Beck2003Mirror}
A. Beck, M. Teboulle: Mirror descent and nonlinear projected subgradient methods for convex optimization. Oper. Res. Lett., 31(3), pp. 167--175, 2003.

\bibitem{article:beck_comirror_2010}
A.~Beck, A.~Ben-Tal, N.~Guttmann-Beck, L.~Tetruashvili: The comirror algorithm for solving nonsmooth constrained convex problems. Operations Research Letters, \textbf{38}(6), pp. 493--498, 2010.

\bibitem{applications_tomography_2001}
A.~Ben-Tal, T.~Margalit, A.~Nemirovski: The Ordered Subsets Mirror Descent Optimization Method with Applications to Tomography. SIAM Journal on Optimization, \textbf{12}(1), pp. 79--108, 2001.	

\bibitem{Boyd2004Subgradient}
S. Boyd, L. Xiao, A. Mutapcic: Subgradient methods. lecture notes of EE392o, Stanford University, Autumn Quarter, 2004(01).

\bibitem{Bubeck_book}
S.~Bubeck: Convex optimization: algorithms and complexity. Foundations and Trends in Machine Learning, \textbf{8}(3--4), pp. 231--357, 2015. \url{https://arxiv.org/pdf/1405.4980.pd}

\bibitem{Chen1993}
G. Chen, M. Teboulle: Convergence analysis of a proximal like minimization algorithm using Bregman functions, SIAM J. Optim. 3 (1993) 538--543.	

\bibitem{article:doan_2019}
T.~T.~Doan, S.~Bose, D.~H.~Nguyen, C.~L.~Beck: Convergence of the Iterates in Mirror Descent Methods. IEEE Control Systems Letters,  \textbf{3}(1), pp. 114--119, 2019.

\bibitem{devolder2013intermediate} 
Devolder, O., Glineur, F., Nesterov, Y.: Intermediate gradient methods for smooth convex problems with inexact oracle (No. UCL-Université Catholique de Louvain). Technical report, CORE-2013017 (2013)

\bibitem{kornilov2023intermediate}
Nikita Kornilov, Mohammad Alkousa, Eduard Gorbunov, Fedor Stonyakin, Pavel Dvurechensky, and Alexander Gasnikov: Intermediate gradient methods with relative inexactness. Journal of Optimization Theory and Applications 207, no. 3 (2025): 62.

\bibitem{Bach2012simpler}
Lacoste-Julien, Simon, Mark Schmidt, and Francis Bach: A simpler approach to obtaining an $O (1/t)$ convergence rate for the projected stochastic subgradient method. arXiv preprint arXiv:1212.2002 (2012).

\bibitem{matyukhin2021convex} 
Matyukhin, V., Kabanikhin, S., Shishlenin, M., Novikov, N., Vasin, A., Gasnikov, A.:  Convex optimization with inexact gradients in Hilbert space and applications to elliptic inverse problems. In International Conference on Mathematical Optimization Theory and Operations Research, Cham: Springer International Publishing, 159-175 (2021)

\bibitem{Nemirovskii1979efficient}
A. Nemirovskii: Efficient methods for large-scale convex optimization problems. Ekonomika i Matematicheskie Metody, 1979. (in Russian)
		
\bibitem{Nemirovsky1983Complexity}
A. Nemirovsky, D. Yudin: Problem Complexity and Method Efficiency in Optimization. J. Wiley \& Sons, New York 1983.	
	
\bibitem{Nesterov_book}
Yu.~Nesterov: Lectures on convex optimization. Switzerland: Springer Optimization and Its Applications, 2018.

\bibitem{article:Nazin_2011}
A.~V.~Nazin, B.~M.~Miller: Mirror Descent Algorithm for Homogeneous Finite Controlled Markov Chains with Unknown Mean Losses. Proceedings of the 18th World Congress The International Federation of Automatic Control Milano (Italy) August 28 - September 2, 2011.
	
\bibitem{article:Nazin_2014}
A.~Nazin, S.~Anulova, A.~Tremba: Application of the Mirror Descent Method to Minimize Average Losses Coming by a Poisson Flow. European Control Conference (ECC) June 24--27, 2014.

\bibitem{polyak1987introduction} 
Polyak, B. T.:  Introduction to Optimization. Optimization Software,  New York (1987)

\bibitem{shor_book}
N. Shor: On the structure of algorithms for the numerical solution of optimal planning and design problems (in Russian) Ph. D. dissertation, Cybernetics Institute, Academy of Sciences of the Ukrainian SSR, Kiev (1964).

\bibitem{stonyakin2021inexact} 
Stonyakin, F., Tyurin, A., Gasnikov, A., Dvurechensky, P., Agafonov, A., Dvinskikh, D., Alkousa, M., Pasechnyuk, D., Artamonov, S.,  Piskunova, V.: Inexact model: A framework for optimization and variational inequalities. Optimization Methods and Software, 36(6), 1155-1201 (2021)

\bibitem{Stonyakin2019some}
Stonyakin, F. S., Alkousa, M. S., Titov, A. A., \& Piskunova, V. V. (2019, June). On some methods for strongly convex optimization problems with one functional constraint. In International Conference on Mathematical Optimization Theory and Operations Research (pp. 82-96). Cham: Springer International Publishing.

\bibitem{Stonyakin2019Adaptive}
Stonyakin, F. S., Alkousa, M., Stepanov, A. N., \& Titov, A. A. (2019). Adaptive mirror descent algorithms for convex and strongly convex optimization problems with functional constraints. Journal of Applied and Industrial Mathematics, 13(3), 557-574.

\bibitem{stonyakin2019}
F. S. Stonyakin, M. Alkousa, A. N. Stepanov, A. A. Titov: Adaptive Mirror Descent Algorithms for Convex and Strongly Convex Optimization Problems with Functional Constraints.  J. Appl. Ind. Math. 13, 557--574 (2019).

\bibitem{vasin2023accelerated}
Vasin, A., Gasnikov, A., Dvurechensky, P., Spokoiny, V.: Accelerated gradient methods with absolute and relative noise in the gradient. Optimization Methods and Software, 1-50 (2023)

\bibitem{Xia2025Lipschitz}
Xia, Yong, Yan-Hao Zhang, and Zhi-Han Zhu: Lipschitz-Free Projected Subgradient Method with Time-Varying Step-Size: Y. Xia et al. Journal of the Operations Research Society of China (2025): 1-9.

\bibitem{Bowen2025Lipschitz}
Yuan, Bowen, and Mohammad S. Alkousa: Lipschitz-Free Mirror Descent Methods for Non-Smooth Optimization Problems. arXiv preprint arXiv:2506.01681 (2025).

\bibitem{Lai2026Capacity}
Yu-Hong Lai, Hao-Chung Cheng: A Mirror-Descent Algorithm for Computing the Petz-Rényi Capacity of Classical-Quantum Channels. arXiv preprint arXiv:2601.10558 (2026).

\end{thebibliography}
\end{document}